\documentclass[a4paper]{article}

\usepackage{amsmath,amssymb,amsthm} 
\usepackage{enumitem}
\usepackage{xcolor}
\usepackage{mathtools}
\usepackage[utf8]{inputenc}
\usepackage{subcaption}

\usepackage{todonotes}

\newcommand{\R}{\mathbb{R}}

\newcommand{\N}{\mathbb{N}}

\newcommand{\dimension}{n}
\newcommand{\Rn}{\R^\dimension}
\newcommand{\diver}{\operatorname{div}}
\newcommand{\Laplace}{\Delta}

\theoremstyle{definition}
\newtheorem{Definition}{Definition}[section]
\newtheorem{Assumption}{Assumption}[section]
\theoremstyle{remark}

\newtheorem{Remark}[Definition]{Remark}
\theoremstyle{plain}
\newtheorem{Proposition}[Definition]{Proposition}
\newtheorem{Lemma}[Definition]{Lemma}
\newtheorem{Theorem}[Definition]{Theorem}
\newtheorem{Corollary}[Definition]{Corollary}

\title{Bounds on Precipitate Hardening of Line and Surface Defects in Solids}
\author{Luca Courte\footnote{Department for Applied Mathematics, University of Freiburg (Germany)} , Kaushik Bhattacharya\footnote{California Institute of Technology, Pasadena, CA (USA)} \ and Patrick Dondl$^*$}
\date{\today}

\begin{document}
	
	\maketitle
	\begin{abstract}
		The yield behavior of crystalline solids is determined by the motion of defects like dislocations, twin boundaries and coherent phase boundaries.  These solids are hardened by introducing precipitates -- small particles of a second phase.  It is generally observed that the motion of line defects like dislocations are strongly inhibited or pinned by precipitates while the motion of surface defects like twin and phase boundaries are minimally affected.  In this article, we provide  insight why line defects are more susceptible to the effect of precipitates than surface defects.  Based on mathematical models that describe both types of motion, we show that for small concentrations of a nearly periodic arrangement of precipitates, the critical force that is required for a surface defect to overcome a precipitate is smaller than that required for a line defect.   In particular, the critical forces for surface and line defects scale with the radius of precipitates to the second and first power, respectively.
	\end{abstract}
	
	\section{Introduction}
	A crystalline solid can deform inelastically through dislocation glide, the motion of twin boundaries as well as the motion of coherent phase boundaries \cite{Christian1995,RobPhillips}.  While the dislocation is a topological line defect, twin and phase boundaries are surfaces across which either the orientation or the structure of crystal changes discontinuously.  Mechanical stress acts as a driving force on these defects, and their motion results in inelastic deformation.  Materials often contain precipitates -- small inclusions of a distinct material (either second phase particles of a different composition or foreign substance), and these affect the motion of both line and surface defects by creating an internal stress field in the material.  Precipitates are often introduced into the material by heat-treatment to inhibit the motion of the defects and thereby increase the yield strength (stress required for inelastic deformation).
	
	In this paper, we will show that the critical external force for a line defect like dislocations required to propagate through an arrangement of precipitates scales with their radius to the power one and the critical external force for a surface defect like a twin or phase boundary scales with the radius to the power two. Hence, for small radii, or, equivalently for small concentrations of precipitates, the effect on surface defects is negligible compared to the effect on line defects.
		
	This result arises by the difference in geometry of the two defects. Indeed, a line defect is a $1$-dimensional object propagating in a given plane through the material while a surface defect is a $2$-dimensional object. It is a well-known fact that the dimensionality of the problem has an impact on the pinning/depinning behavior of interfaces, this was first studied for charged density waves (e.g., \cite{PhysRevB.19.3970, PhysRevLett.68.670, PhysRevLett.50.1486}) and later for magnetic domain walls and general interface motion following the quenched Edwards-Wilkinson equation, see for instance \cite{Feigel'man1983, PhysRevB.32.280, PhysRevA.45.R8309, PhysRevB.48.7030, PhysRevLett.52.1547}.   The arguments in these works are generally heuristic.  By contrast, in this work we prove mathematically rigorous bounds on the critical depinning threshold---by constructing appropriate viscosity sub- and supersolutions---for line- and surface-like defects and then apply these bounds to physical models.
	
	We emphasize that the dimensional argument has to do with pinning rather than the possibility that twin boundaries  always hit the precipiates while dislocations might miss some.  We show in  Lemma 3.4 that any plane will almost surely intersect some precipitate; consequently any dislocation gliding on a plane almost surely encounters some precipitate.

	Our result has some very interesting implications.  In any crystal, the energetics and mobility of dislocations and twin boundaries depend on crystallography.  In high symmetry crystals like copper or aluminum, symmetry dictates that the system with the lowest critical resolved stress is sufficient to accommodate all deformations.  However, in low symmetry materials  like magnesium and zirconium, this is not the case and therefore one sees multiple defects.  Magnesium and its alloys have been the topic of much recent interest since they have potentially the highest strength to weight ratio.  However, they lack ductility.  In magnesium, the so-called basal dislocation is an order of magnitude softer than other defects, but insufficient to accommodate arbitrary distortions.  So it is common to see twins, especially in tension \cite{Hull1917,ReedHill1957,Christian1995}.   Further this significant anisotropy is believed to be ultimately responsible for the low ductility.  The results here suggest that precipitate hardening can have a differential effect and this can be used to improve the strength and ductility of magnesium.  Indeed, precipitate hardening is used extensively in magnesium alloys.  It has been observed through neutron diffraction and modeling that the critical stress in the basal system increases three-fold while that for tensile twinning remains essentially unchanged during aging in Mg-Y-Nd-Zr alloys \cite{agnew2013}.  It is important to note here that observations in other related alloys do not show such a clear distinction due to the elongated shape and basal orientation of the precipitates as well as the fact that twin growth is accompanied by basal slip \cite{robson2011,stanford2012}.   Precipitates play a similar role in low stacking fault steels like TWIP steels where they increase yield strength by inhibiting dislocation motion and leave hardening rate that is influenced by twinning unaffected \cite{Chateau2010}.  
	
	The commonly used shape-memory alloy has two inelastic deformation modes, plasticity due to dislocations and superelasticity due to stress induced phase transformations.  The widely used shape-memory alloy nickel-titanium undergoes plastic deformation at extremely low stress, and this hides its useful superelastic effect.  Therefore, commercial alloys are precipitate hardened.  They increase the plastic yield strength by inhibiting dislocation activity but leave superelasticity governed by phase and twin boundaries unaffected \cite{otskuka}.
	
	\section{Model and Results}

\subsection{Model}
	We describe both defects, which are one (dislocations) and two (twin boundaries) dimensional subsets of $\R^3$ as  graphs of suitable functions and then work with the evolution equations of these functions. 
	
	The {\it evolution of a twin boundary} $\Gamma_{twin}(t) \coloneqq \{ (x, w(x, t)) \;|\; x\in \R^2 \}$ can be described by a non-local version of a quenched Edwards-Wilkinson (QEW) equation\footnote{We denote by QEW an evolution equation, driven by an external force, with linearized line (or surface) tension in a random medium with finite correlation length.  We refer to (\ref{eq: driving equation twin boundary}) as QEW-1/2 since we have a half-Laplace operator}
	\begin{align}\label{eq: driving equation twin boundary}
	\partial_t w(x, t) = -(-\Laplace)^{1/2} w(x, t) - \varphi(x, w(x, t)) + F,
	\end{align}
	where $\partial_t$ is the derivative with respect to time and $-(-\Laplace)^{1/2}$ is the half Laplacian with respect to space (see  \cite{Dondl2016}).
	\begin{figure}
		\centering
		\begingroup%
		\makeatletter%
		\setlength{\unitlength}{215bp}%
		\makeatother%
		\begin{picture}(1,0.4)%
		\put(0,0.05){\includegraphics[width=\unitlength]{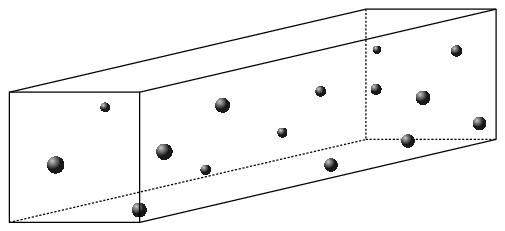}}%
		\put(0.12,0.00){\color[rgb]{0,0,0}\makebox(0,0)[lb]{\smash{$x_1$}}}%
		\put(-0.05,0.15){\color[rgb]{0,0,0}\makebox(0,0)[lb]{\smash{$x_2$}}}%
		\put(0.6,0.08){\color[rgb]{0,0,0}\rotatebox{13.219378}{\makebox(0,0)[lb]{\smash{$y$}}}}%
		\end{picture}%
		\endgroup%
		\caption{A part of the infinite strip, i.e., $\mathbb{T}^2 \times \R$. The spherical objects are a representation of the precipitates. We model these by the function $\varphi$.}
		\label{fig:1}
	\end{figure}
	
	In this work, we consider a {\it line  tension model for a dislocation} \cite{RobPhillips}. We assume that the dislocation is confined to a random glide plane $\pi = \{ x_2 = \varpi \}$, see Figure \ref{fig:2}. Then the evolution of the dislocation $\Gamma_{dis}(t) \coloneqq \{ (x, \varpi, v(x, t)) \;|\; x \in \R \}$ is described by the following partial differential equation
	\begin{align}\label{eq: driving equation dislocation}
	\frac{\partial_t v(x, t)}{\sqrt{1+|\nabla v(x, t)|^2}} =  \diver\left(\frac{ \nabla v(x, t)}{\sqrt{1+|\nabla v(x, t)|^2} }\right) - \tilde{\varphi}(x, v(x, t)) + F,
	\end{align}
	where $\tilde{\varphi} = \varphi(\cdot, \varpi, \cdot)$ (see also \cite{pinning_interfaces_random_media} and \cite{Craciun2004}). Note that by confining the dislocation to the glide plane, we have ignored climb.
\vspace{\baselineskip}	
	 
	\begin{figure}
		\begin{subfigure}{0.45\textwidth}
			\centering
			\begingroup%
			\makeatletter%
			\setlength{\unitlength}{120bp}%
			\makeatother%
			\begin{picture}(1,0.4)%
			\put(0,0.05){\includegraphics[width=\unitlength]{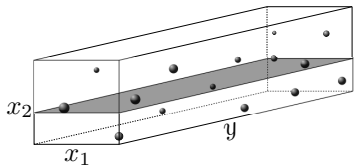}}%
			\put(0.1,0.00){\color[rgb]{0,0,0}\makebox(0,0)[lb]{\smash{$x_1$}}}%
			\put(-0.08,0.15){\color[rgb]{0,0,0}\makebox(0,0)[lb]{\smash{$x_2$}}}%
			\put(0.6,0.08){\color[rgb]{0,0,0}\rotatebox{13.219378}{\makebox(0,0)[lb]{\smash{$y$}}}}%
			\end{picture}%
			\endgroup%
			\caption{One possible slip plane  $\pi = \{ x_2 = \varpi \}$ for a dislocation is shown in this figure. Note that the plane intersects some precipitates.}
			\label{fig:2a}
		\end{subfigure}
		\hspace{0.05\textwidth}
		\begin{subfigure}{0.45\textwidth}
			\centering
			\begingroup%
			\makeatletter%
			\setlength{\unitlength}{50bp}%
			\makeatother%
			\begin{picture}(0.7,1)%
			\put(0,0.1){\includegraphics[height=\unitlength]{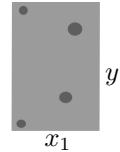}}%
			\put(0.25,0.0){\color[rgb]{0,0,0}\makebox(0,0)[lb]{\smash{$x_1$}}}%
			\put(0.7,0.5){\color[rgb]{0,0,0}\makebox(0,0)[lb]{\smash{$y$}}}%
			\end{picture}%
			\endgroup%
			\caption{A top-down view of the plane $\pi$ from Figure \ref{fig:2a}. The precipitates are highlighted in dark gray. As not all precipitates are cut in the middle, the resulting radii of the disks in the plane vary.}
		\end{subfigure}
		\caption{These figures show how the slip planes are introduced in the model of the crystal.}
		\label{fig:2}
	\end{figure}

\noindent In both cases, the driving equation  depends on three terms. The first term is a penalty for the deviation of the geometry of the defect from a flat state (which in our setting also ensures that the graph-setting remains appropriate), the second term  -- where $\varphi \colon \R^3\to\R$ is assumed to be bounded and uniformly Lipschitz-continuous function -- describes the interaction of the graph with the precipitates, and the last term constitutes the external driving force.  In the context of the twin boundary, the first term arises from elasticity \cite{Dondl2016} (also \cite{Koslowski2002}).  The actual interaction between a defect and a precipitate is non-local, hence the second term should be a non-local potential. However, following  \cite{Dondl2016} we assume that this interaction can be approximated well by the local term (see also the discussion of M. Koslowski et al. \cite{Koslowski2002}).  	In both models, the constants such as elastic parameters, line tension, etc.~have been suppressed. As we are merely interested in a scaling result, this suppression does not affect the comparison of critical forces.
	
We assume $w=0$ and $v=0$ as initial conditions. We furthermore assume $F\ge 0$ and $\varphi\ge 0$, which in our model implies that the precipitates always impede the motion in the (positive) $y$-direction that the external driving force is favoring.

	\begin{Remark}
		We choose the same interaction potential for both equations to keep the situations as similar as possible. Importantly, the scaling results remain valid if one does consider two different interaction potentials as long as they are both localized around the centers of the precipitates (see Assumption \ref{ass: distribution of precipitates}) and have a ``radius'' that scales linearly in $R$ (see Assumption \ref{ass: shape of precipitates}).
	\end{Remark}
	
	\begin{Remark} Let us briefly comment why we decided to consider a linear, but non-local model for twin boundaries and a nonlinear, but local model for dislocations. We use a nonlinear expression (i.e., the mean curvature) to penalize the deviation from a flat state for dislocations, while we use a linearized expression for twin boundaries. One reason for this is of a technical nature: the nonlinear equation for twin boundaries \cite{Abeyaratne1990} does not admit a comparison principle and even existence results are unavailable without further regularization \cite{Dondl:2010kw}. Furthermore, it is observed that twin boundaries do not usually exhibit a strong curvature (see, e.g., \cite{Abeyaratne:1996vv}), especially compared to dislocation lines (see, e.g., \cite{robertson_disl}) and so we consider our linear approximation to be suitable here.   Moreover, our results do in fact remain unchanged if both equation \eqref{eq: driving equation twin boundary} and equation \eqref{eq: driving equation dislocation} are both replaced by their linear, local analogue which we refer to  as QEW-1
	\begin{align}\label{eq: qew1}
	\partial_t v(x, t) =  \Delta v(x, t) - \tilde{\varphi}(x, v(x, t)) + F,
	\end{align}
 This is proved in section \ref{sec: qew}.  Our main result can thus be interpreted in the sense that the difference in scaling is only a matter of dimension. It is unaffected by the nature of the respective operator---at least when considering powers $\alpha$ of the Laplacian for $\alpha\ge 1/2$.
	\end{Remark}
	
\subsection{Background}

	Note, that since $\varphi \colon \R^3\to\R$ is assumed to be bounded and uniformly Lipschitz-continuous which allows to use exponential time scaling, i.e., replace a subsolution $\underline{v}$ and a supersolution $\overline{v}$ of \eqref{eq: driving equation dislocation} by $\underline{V}\coloneqq e^{-\lambda t}\underline{v}$ and $\overline{V}\coloneqq e^{-\lambda t}\overline{v}$ respectively, to derive a comparison principle. These new functions are sub- and supersolution to an equation $\partial_t W = H(x, t, W, \nabla W, D^2 W)$ with an appropriate choice of $H$. By choosing $\lambda$ wisely, the righthandside will be -- in the nomenclature of \cite{Crandall1992} -- proper and hence we have a comparison principle. This does imply a comparison principle for \eqref{eq: driving equation dislocation} and we can conclude that there exists a unique viscosity solution provided that the initial datum is smooth enough, see \cite[Theorem~8.2, Theorem~4.1]{Crandall1992}. For equation \eqref{eq: driving equation twin boundary} a similar argument can be found in \cite[Theorem~2]{Imbert2005} and for further details we refer to \cite{Droniou2006}. We conclude, that both equations \eqref{eq: driving equation twin boundary}, \eqref{eq: driving equation dislocation} satisfy a comparison principle and hence unique viscosity solutions exist.

	We are interested in the pinning of defects by precipitates, i.e, the question whether 
	\begin{enumerate}
		\item there exists stationary supersolutions $\overline{w} \colon \R^2 \to \R$, $\overline{w}\ge 0$ or $\overline{v} \colon \R \to \R$, $\overline{v}\ge 0$ such that, respectively, $w(x,t) \le \overline{w}(x)$ or $v(x,t) \le \overline{v}(x)$ for all $t\ge 0$, $x\in\R^{2,3}$. 
		\item or whether $w$ or $v$ are unbounded as $t\to\infty$ due to the existence of propagating subsolutions (e.g., $\underline{w}(x,t)$, such that $w(x,t) \ge \underline{w}(x,t)$ with $\underline{w}(x,0)=0$, $\underline{w}(x,t)\ge ct$ for all $t\ge 0$, $x\in\R^2$ and some constant $c\ge 0$). 
	\end{enumerate}
	
	While the question whether the two points above form a dichotomy is open in the general setting \cite{Dondl_16e, Bodineau:2013ur} \cite{Dirr2006, Coville2010}, the following simple statement follows immediately from the comparison principle using the assumptions on $\varphi$ made above.
	\begin{Proposition}\label{prop: ex pinning forces}
		There are critical forces $\overline{F}_{twin} \ge 0$, $\underline{F}_{twin} \ge 0$ such that for all $F < \overline{F}_{twin}$ the interface $\Gamma_{twin}(\cdot)$ gets pinned, i.e., for all $F < \overline{F}_{twin}$, there is a stationary supersolution. Moreover, for all $F > \underline{F}_{twin}$ the interface does not get pinned, i.e., there is a propagating subsolution. The same result holds with a critical forces $\overline{F}_{dis}$, $\underline{F}_{dis}$ for dislocations.
	\end{Proposition}
	
	Our strategy of proving that the pinning threshold for twin boundaries is lower than the pinning threshold for dislocations involves obtaining a lower bound for $\overline{F}_{dis}$ and an upper bound for $\underline{F}_{twin}$ and then comparing these bounds to conclude. To retrieve the bounds, we  construct worst case scenarios for pinning (in the case of dislocations) and depinning (in the case of twin boundaries).
	
	It is clear that the bounds will depend on the pinning potential $\varphi$.  In real crystals, it arises from precipitates that have many different arrangements (random, periodic, planar) and  shapes (rods, discs, spheres, faceted).  Therefore, we take $\varphi$ to be given by (a regularization of) the characteristic function of the precipitates.  Further, we assume that the distribution is (periodic orthogonal to the propagation direction and well spaced in the propagation direction, while the shape is bounded by a cube from the inside and the outside.

\subsection{Main results}
	We first assume that the domain is an infinite strip $\Omega \coloneqq \mathbb{T}^2 \times \R$, i.e., we assume periodicity orthogonal to the propagation direction. The coordinates of the torus are denoted by $x_1$ and $x_2$ and the coordinates of $\Omega$ by $x_1$, $x_2$ and $y$. We further use the convention $\mathbb{T}^2 \cong [-1, 1]^2$. Moreover, let $\beta > 0$, $\lambda \in (0, 1]$, $0 < \varphi_* \le \varphi^* < \infty$ be fixed. For each $R \in (0, \tfrac{1}{2})$ we consider a distribution of precipitates in the strip.
	
	\begin{Assumption}\label{ass: distribution of precipitates}
		Let $(X_1^i, X_2^i, Y^i)_{i\in \N}$ be a family of random variables, that represent the centers of the precipitates. We will assume that
		\begin{enumerate}[label=\alph*)]
			\item $(X_1^i)_{i\in\N}, (X_2^i)_{i\in \N}$ are identically and independently distributed with $X_1^1 \sim \mathrm{Unif}([-1, 1])$,
			\item The random variables $|Y^i| = \operatorname{dist}(0, Y^i)$ have finite expected value and $|Y^i(\omega)-Y^j(\omega)| \ge 2R^{1-\beta}$ if $i\not= j$ almost surely.
		\end{enumerate}
	\end{Assumption}
	Furthermore, we will also make the following assumption on the shape of the precipitates.
	\begin{Assumption}\label{ass: shape of precipitates}
		For each $i \in \N$ consider a smooth function $\psi_i : \R^3 \to \R$ with 
		\[
		\varphi_*\chi_{[-\lambda R, \lambda R]^3} \le \psi_i \le \varphi^*\chi_{[-R, R]^3},
		\]
		i.e., the precipitates contain a  small cube and are bounded by a cube of sidelength $R$ and have a pinning strength which is bounded by $\varphi_*$ and $\varphi^*$. 
	\end{Assumption}
	We assume that the resistance provided by the precipitates is given by
	\[
	\varphi(\cdot, \omega) \coloneqq \sum_{i=1}^\infty \varphi^i \psi_i(\cdot - (X^i_1(\omega), X^i_2(\omega), Y^i(\omega))).
	\]
	
	\begin{Remark}\label{rmk:assumptions}
		The assumptions are chosen in such a way, that the $x_1$ and $x_2$ components of the centers of the contained balls are independent and identically distributed (iid), and that the distance between two precipitates cannot approach zero too fast. Our proofs will work for any configuration of precipitates that satisfy these conditions. Even though not included in our assumptions, one could image toroidal precipitates for which our results will also hold. For clarity, we chose to formulate our assumptions in the way above and they do include most of the physical cases as spherical, elliptical and rod-shaped precipitates.
	\end{Remark}
	
We have the following results.
	
	\begin{Theorem}[Bounds on the critical force for dislocation motion]\label{satz: critical force disloc}
		For $R > 0$ small enough, we almost surely have the following bounds
		\[
		\varphi^* R \ge \underline{F}_{dis} \ge \overline{F}_{dis} \ge  \min\{\varphi_*,2(1-\lambda R)^{-2}\}\cdot \lambda R
		, 
		\]
		for the critical pinning force for dislocations confined to a random glide plane $\pi = \{ x_2 = \varpi \}$. This means that the critical force scales with the first power of the radius of the precipitates.
	\end{Theorem}
	\begin{Theorem}[Bounds on the critical force for twin boundary motion]\label{satz: critical force twin}
	For $R > 0$ small enough, the critical pinning force for twin boundaries can almost surely be estimated by
	\[
	\varphi^* R^2 \ge \underline{F}_{twin} \ge
	\overline{F}_{twin} \ge \min\{ \varphi_*, \tfrac{1}{2C} \}\cdot (\lambda R)^2,
	\]
	where $C > 0$ is a geometric constant. Hence, the critical force scales with the second power of the radius of the precipitates.
	\end{Theorem}
	The proof of these theorems are mainly based on two ingredients. First, the comparison principle plays a vital role as we construct a stationary supersolution and a non-stationary subsolution. Second, the geometry of the problem is crucial as we have to construct said sub-/supersolution. Indeed, the first and second power of $R$ come from the $(1+1)$- and $(2+1)$-dimensional setting. We note that this problem is related to geometry and that there are no hidden effects stemming from the non-linearity of the mean-curvature equation, the non-locality of the fractional Laplacian, or the dimension. To substantiate this claim, we prove the scaling result for $n$-dimensional QEW equations in section \ref{sec: qew}.
	
	We can combine the last two theorems in order to obtain the following result.
	\begin{Corollary}
		For all $R \le \tfrac{1}{2}$ with distributions of precipitates that satisfy assumptions \ref{ass: distribution of precipitates} and \ref{ass: shape of precipitates}, there is almost surely a constant $c_0 > 0$ depending only on $\lambda$ and the pinning strength, i.e., $\varphi_*$ and $\varphi^*$, such that
		\[
			c_0 R \le \frac{\underline{F}_{twin}}{\overline{F}_{dis}} \le c_0^{-1} R.
		\]
		In particular, there exists $R > 0$ small enough, such that
		\[
		\underline{F}_{twin} < \overline{F}_{dis}
		\]
		holds almost surely.
	\end{Corollary}
	This means that if the concentration of the precipitates is small enough, there are external forces $F$, such that dislocations get blocked while twin boundaries can move freely throughout the crystal.
	\begin{proof}
		We can apply theorem \ref{satz: critical force disloc} and \ref{satz: critical force twin} to obtain the following inequalities
		\[
		\frac{\min\{\varphi_*, \tfrac{1}{2C}\}\cdot \lambda^2 R^2}{\varphi^* R} \le \frac{\underline{F}_{\textit{twin}}}{\overline{F}_{\textit{dis}}} \le \frac{ \varphi^* R^2 }{ \min\{\varphi_*,2(1-\lambda R)^{-2}\}\cdot \lambda R }.
		\]
		The statement follows by choosing $c_0$ depending on $\lambda, \varphi_*, \varphi^*$ and the geometric constant $C$.
	\end{proof}
	
	\section{Proofs}
	Both results are a consequence of the comparison principle and can be proven with similar techniques. Therefore, we start by deriving the result for the abstract equation
	\begin{equation}\label{eq: abstract eq}
	N[u] u_t = A[u] + \Phi(\cdot, u) + F \text{ in } \mathbb{T}^n \times [0, \infty),
	\end{equation}
	with zero initial condition, where $n > 0$ is the dimension of the interface, $N[\cdot]$ is an operator that is invariant under the addition of constants with $N[0] = 1$ and $ 0 < N[u](x) \le 1$  for all $x \in \mathbb{T}^n$ and $A[\cdot]$ is an operator that is invariant under the addition of constants with $A[0] = 0$, $\Phi : \mathbb{T}^n \times \R \to (-\infty, 0]$ and $F > 0$ some positive force. We assume that a viscosity solution exists and the comparison principle holds for this equation.
	
	The following two theorems from which we will derive theorem \ref{satz: critical force disloc} and \ref{satz: critical force twin} are based on the existence of sub- and supersolutions to the following partial differential equation
	\begin{equation} \label{eq: abstract simplif}
	0 = A[u](x) - \mu \chi_{[-\rho, \rho]^n}(x) + F_0 \text{ in } \mathbb{T}^n,
	\end{equation}
	where $\mu > 0$, $\rho \in (-1, 1)$ and $F_0 > 0$.
	\begin{Proposition}\label{prop: 3.1}
		Assume that there are $\rho \in (-1, 1)$, $\mu > 0$ and a point $y_0 \in \R$ such that $\Phi(\cdot, y_0 + s) \le -\mu \chi_{[-\rho, \rho]^n}(\cdot)$ for all $s \in (-\rho, \rho)$. If there is a $\overline{F} > 0$ such that for $F_0 \coloneqq \overline{F}$ there is a viscosity supersolution $v_0 : \mathbb{T}^n \to \R$ to equation \eqref{eq: abstract simplif} with 
		$$
		\max_{x\in \mathbb{T}^n} |v_0(x)| < \rho,
		$$then for all $F \le \overline{F}$ there is a stationary supersolution to \eqref{eq: abstract eq}.
	\end{Proposition}
	\begin{proof}
		Define $v(x, t) \coloneqq v_0(x) + y_0$, then we have for all $F \le \overline{F} = F_0$ that
		\begin{align*}
			N[v(x, t)] v_t(x, t) = 0 
			&\ge A[v_0](x) - \mu \chi_{[-\rho, \rho]^n}(x) + F_0 \\
			&\ge A[v(\cdot, t)](x) + \Phi(x, v(x, t)) + F
		\end{align*}
		as $v_0$ is a viscosity supersolution and $|v_0(x)|\le \rho$.
	\end{proof}
	
	\begin{Proposition}\label{prop: 3.2}
		Assume that there are $\rho \in (-1, 1)$, $\mu > 0$ such that $$\Phi(x, y) \ge - \sum_{i=1}^\infty \mu \chi_{(-\rho, \rho)^{n+1}}(x - x_i, y - y_i),$$ where $x_i \in \mathbb{T}^n$ and $y_i \in \R$ with $\rho < y_i \le y_{i+1}$ and $|y_i - y_j| \ge 2\rho^{1-\beta}$ for all $i, j \in \N$ with $i\not=j$ and some $\beta \in (0, 1)$. If there is a $\underline{F} > 0$ such that for $F_0 \coloneqq \underline{F}$ there is a viscosity subsolution $w_0 : \mathbb{T}^n \to \R$ to equation \eqref{eq: abstract eq} with
		\begin{equation}\label{eq: distortion}
			\max_{x\in \mathbb{T}^n}\{ w_0(x)\} - \min_{x\in\mathbb{T}^n}\{ w_0(x)\} < 2\rho^{1-\beta} -2 \rho 
		\end{equation}
		then for all $F > \underline{F}$ there is a non-stationary solution $w$ to \eqref{eq: abstract eq} with $\lim_{t\to \infty} w(x, t) = +\infty$.
	\end{Proposition}
	\begin{proof}
		Let $F > \underline{F}$ and let $w : \mathbb{T}^n \times [0, \infty) \to \R$ be the viscosity solution to \eqref{eq: abstract eq} with zero initial condition. Until $w$ hits the first obstacle, it propagates like a flat plane with velocity $F$ (note that $N[Ft]=1$). Let $t_1$ be such that $w(x, t_1) = y_1 - \rho$, i.e., the last time before $w$ hits the first precipitate. Now, we discuss why $w$ passes through this precipitate. Using a translation, we can assume without loss of generality that $w(\cdot, t_1) = 0$.
		Set $\tau \coloneqq F-\underline{F}$ and define $W(x, t) \coloneqq -\max_{x\in\mathbb{T}^n}\{w_0(x)\} + w_0(x) + \tau t$, then it holds $W(\cdot, 0) \le w(\cdot, t_1)$ and
		\begin{align*}
			N[W]W_t - A[W] - \Phi(x, W) - F &\le -(A[W] + \Phi(x, W) + F - \tau) \\
			&\le -(A[w_0] - \mu \chi_{(-\rho, \rho)^{n}}(x)  + F_0) \\
			&\le 0,
		\end{align*}
		if $W$ does not interact with a second precipitate. Due to \eqref{eq: distortion} this is guaranteed and hence $W \le w$ and $w$ passes through the precipitate. Moreover, there is a time point $t_2$, where $W(\cdot, t_2)$ does not cross any precipitates and $W$ has passed the first one. Now, we can apply a translation and construct, in the same way, a solution that passes through the second precipitate and has an initial value that lies below $w(\cdot, t_2)$. Again the comparison principle shows that $w$ has to pass through the second precipitate. Repeating this argument shows that $w$ crosses every precipitate.
	\end{proof}

	\subsection{Bounds for the pinning threshold of dislocations}
	In order to apply proposition \ref{prop: 3.1} or proposition \ref{prop: 3.2} we have to construct a sub- or a supersolution to equation \eqref{eq: abstract simplif}. The construction is based on ideas in \cite{pinning_interfaces_random_media}.
	
	\begin{Lemma}\label{lem: disloc aux}
		Let $\rho > 0, \mu > F_0 \ge 0$ then the function $v :\mathbb{T} \to \R$ given by
		\[
			v(x) \coloneqq \left\{\begin{array}{ll} 
			v_{in}(x) &\text{ in } (-\rho, \rho) \\
			v_{out}(x) &\text{ in } (-1, -\rho] \cup [\rho, 1),
			\end{array}\right.
		\]
		with 
		\begin{align*}
		v_{in}(x) &\coloneqq -\sqrt{(\mu - F_0)^{-2} - x^2} + \sqrt{(\mu-F_0)^{-2} - \rho^2}, \\
		v_{out}(x) &\coloneqq \sqrt{F_0^{-2}-(1-|x|)^2} - \sqrt{F_0^{-2}-(1-\rho)^2}
		\end{align*}
		is well-defined if
		$$
			\mu-\rho^{-1} \le F_0 \le (1-\rho)^{-1}.
		$$
		Moreover, $v$ is a viscosity supersolution if 
		$
			F_0 \le \rho \mu
		$
		and a viscosity subsolution if
		$
			F_0 \ge \rho \mu.
		$
		Finally, $v$ satisfies the following inequalities
		\[
			\max_{x\in\mathbb{T}}|v(x)| \le \rho,
		\]
		if $\mu-2\rho^{-1} \le F_0 \le 2 \rho (1-\rho)^{-2}$ and
		\[
			\max_{x\in \mathbb{T}^n}\{ v_0(x)\} - \min_{x\in\mathbb{T}^n}\{ v_0(x)\} < 2\rho^{1-\beta} - 2\rho,
		\]
		if $F_0 < \frac{4\rho^{1-\beta}-4\rho-\mu\rho^2}{1-2\rho}$.
	\end{Lemma}
	\begin{proof}
		The well-definiteness follows by a simple calculation. Moreover, due to this construction $v$ statisfies equation \eqref{eq: abstract simplif} pointwise for all $x\in [-1, 1]$ with $|x| \not= \rho$. Now, to make $v$ a viscosity sub- or supersolution the mean curvature at this point has to be negativ or positive definite. Due to the symmetry of $v$ this leads to the condition $-v_{out}(\rho) \le v_{in}(\rho)$ for $v$ being a supersolution and $-v_{out}(\rho) \ge v_{in}(\rho)$ for $v$ being a subsolution. A simple computation leads to the asserted bounds.
		
		For the estimates on the maximum of $|v|$, just note that $$\max|v| \le \max\{ -v_{in}(0), v_{out}(1) \}$$ and we can use the local Lipschitz-continuity of the square root to estimate
		\[
			-v_{in}(0) \le \tfrac{1}{2 \sqrt{(\mu-F_0)^{-2}}}\rho^2 = \tfrac{1}{2} (\mu - F_0) \rho^2
		\]
		and
		\[
		v_{out}(1) \le \tfrac{1}{2 \sqrt{F_0^{-2}}}(1-\rho)^2 = \tfrac{1}{2} F_0 (1-\rho)^2.
		\]
		If these quantities should be less then $\rho$ then $F_0$ has to obey the stated bounds.
		For the final statement, we compute
		\[
		\max_{x\in \mathbb{T}^n}\{ v_0(x)\} - \min_{x\in\mathbb{T}^n}\{ v_0(x)\} = v_{out}(1) - v_{in}(0) \le \tfrac{1}{2}( F_0 - 2\rho F_0 + \mu \rho^2 ).
		\]
		Now, if $F_0  < \frac{4\rho^{1-\beta}-4\rho-\mu\rho^2}{1-2\rho}$ then
		\[
			\tfrac{1}{2}( F_0 - 2\rho F_0 + \mu \rho^2 ) < 2\rho^{1-\beta} - 2\rho,
		\]
		and the statement follows.
	\end{proof}

	We can now prove the upper bound of theorem \ref{satz: critical force disloc}.
	\begin{proof}[Proof of theorem \ref{satz: critical force disloc} (upper bound)]
		For almost any random glide plane $\{x_2 = \varpi\}$, we have $\tilde{\varphi}(x, y) = \varphi(x, \varpi, y)$ for $x, y \in \R$, and therefore the distance between the centers of two precipitates is always bigger then $2R^{1-\beta}$ by assumption \ref{ass: distribution of precipitates}.
		
		Hence, we can apply proposition \ref{prop: 3.2} with $\rho = R$ and $\mu = \varphi^*$ with an $\underline{F}$ that satisfies the following properties, see lemma \ref{lem: disloc aux},
		\begin{align*}
			\max\left\{\varphi^* - R^{-1}, \varphi^*R  \right\} \le \underline{F} \le \min\left\{\frac{4R^{1-\beta}-4R-\varphi^*R^2}{1-2R}, (1-R)^{-1}\right\}.
		\end{align*}
		If $R$ is small enough, such an $\underline{F}$ exists, as
		\[
			\varphi^*-R^{-1} \le \varphi^* R \le \frac{4R^{1-\beta}-4R-\varphi^*R^2}{1-2R} \le (1-R)^{-1},
		\]
		where the second inequality holds as $\beta \in (0, 1)$. Now, we can choose $\underline{F} \coloneqq \varphi^*R$ and apply proposition \ref{prop: 3.2} to see that for all $F > \underline{F}$ the solution to \eqref{eq: driving equation dislocation} crosses all precipitates. Hence, the critical depinning force satisfies $\underline{F}_{dis} < \underline{F}$.
	\end{proof}
	
	For the proof of the lower bound, we have to ensure that the random plane intersects a precipitate.
	\begin{Lemma}\label{lem: we find almost surely a precipiate of radius eta r in a random plane}
		Let $\varpi \in [-1, 1]$ and $\pi \coloneqq \{x_2 = \varpi \}$. Then almost surely the plane $\pi$ intersects at least a precipitate with an intersection containing a square of side-length $2r$.
	\end{Lemma}
	\begin{proof}
		The probability that the plane $\pi$ intersects at least a precipitate with an intersection containing a square of side-length $2r$ is $r$. Moreover this probability, is bigger then the probability that the plane $\pi$ intersects infinitely many precipitates with an intersection containing a circle of side-length ${2r}$. Let us denote this event by $A$. As the $X^i_2$ are independent so are the events $A_i$. Moreover, we have
		\[
		\sum_{i=1}^\infty P(A_i) = \infty.
		\]
		Hence the Borel-Cantelli Lemma applies and it follows that $P(A) = 1$.
	\end{proof}
	Now, that we know that the random plane intersects almost surely a precipitate, we can prove the lower bound.
	\begin{proof}[Proof of theorem \ref{satz: critical force disloc} (lower bound)]
		By lemma \ref{lem: we find almost surely a precipiate of radius eta r in a random plane}, we have almost surely that $\tilde{\varphi}(x, y) \le -\varphi_* \chi_{[-r, r]^2}(x, y)$
		Therefore, we can apply proposition \ref{prop: 3.1} with $\rho = \lambda R$ and $\mu = \varphi_*$ with an $\overline{F}$ that satisfies the following properties, see lemma \ref{lem: disloc aux},
		\begin{align*}
		\varphi^* - (\lambda R)^{-1} \le \overline{F} \le \min\left\{\varphi_* \lambda R, 2\lambda R(1-\lambda R)^{-2}, (1-r)^{-1}\right\}.
		\end{align*}
		If $R$ is small enough such an $\overline{F}$ exists. Now, we can choose $$\overline{F} \coloneqq \min\{\varphi_*, 2(1-\lambda R)^{-2}\} \cdot \lambda R$$ and apply proposition \ref{prop: 3.2} to see that for all $F > \overline{F}$ the solution to \eqref{eq: driving equation dislocation} remains bounded. Hence, the critical depinning force satisfies $\overline{F}_{dis} \ge \overline{F}$.
	\end{proof}

	\subsection{Bounds for the pinning threshold of twin boundaries}
		We want to establish bounds for the pinning threshold of twin boundaries. In contrast to equation \eqref{eq: driving equation twin boundary}, we are looking at the following more general partial differential equation
	\[
	\partial_t w(x, t) = -(-\Laplace)^\alpha w(x, t) - \varphi(x, w(x, t)) + F ~~\text{ in } \mathbb{T}^2 \times [0, \infty).
	\]
	This nonlocal partial differential equation obeys a comparison principle \cite{Droniou2006}.
	Based on ideas established in \cite{Dondl2015}, we are going to construct the solution to \eqref{eq: abstract simplif}. 
	\begin{Lemma}\label{lem: solution fraclp u = g}
		Let $F_1, F_2 > 0, \rho > 0$, such that $g \coloneqq F_2 -(F_1+F_2) \chi_{[-\rho, \rho]^2}$ has vanishing average over $[-1, 1]^2$. Then the periodic solution with vanishing average $u$ of
		\[
		(-\Laplace)^\alpha u(x) = g(x) \text{ in } [-1, 1]^2
		\]
		is given by
		\[
		u(x_1, x_2) \coloneqq -\frac{4}{\pi^{2+2\alpha}}\sum_{n, m \in \N} \frac{F_1 + F_2}{(n^2+m^2)^\alpha nm} \sin(\pi n\rho)\sin(\pi m\rho) \cos(\pi n x_1 + \pi m x_2).
		\]
	\end{Lemma}
	\begin{proof}
		For every $n, m \in \N$ denote by $s_{n, m} \coloneqq \sin(\pi nx_1+\pi mx_2)$ and $c_{n, m} \coloneqq \cos(\pi nx_1+\pi mx_2)$ the eigenfunctions of the negative Laplacian in $[-1, 1]^2$ with periodic boundary conditions. Moreover we define the eigenvalue to $s_{n, m}$ and $c_{n, m}$ as $\lambda_{n, m} \coloneqq \pi^2(n^2+m^2)$. We are now going to compute the Fourier series of $g$.  Due to the symmetries of $g$, we have $\left<g, s_{n, m}\right>_{L^2} = 0$. Moreover, note that
		\begin{align*}
		\int_{-a}^a \int_{-a}^a c_{n, m}(x, y) \;\mathrm{d}x \;\mathrm{d}y &= \frac{2}{\pi^2 n m} (\cos(\pi an - \pi am)-\cos(\pi an + \pi am)) \\
		&= \frac{4}{\pi^2 n m} \sin(\pi an) \sin(\pi am),
		\end{align*}
		which allows us to compute
		\begin{align*}
		\left<g, c_{n, m}\right>_{L^2} &= F_2 \underset{=0}{\underbrace{\int_{-1}^1\int_{-1}^1 c_{n, m}(x_1, x_2) \;\mathrm{d}x_1 \;\mathrm{d}x_2}} - (F_1 + F_2)  \int_{-\rho}^\rho\int_{-\rho}^\rho c_{n, m}(x_1, x_2) \;\mathrm{d}x_1 \;\mathrm{d}x_2\\
		&= - \frac{4(F_1 + F_2)}{\pi^2 n m} \sin(\pi n \rho) \sin(\pi m\rho).
		\end{align*}
		Hence the Fourier series of $g$ is given by
		\[
		g(x) \coloneqq - \sum_{n, m \in \N} \frac{4(F_1 + F_2)}{\pi^2 n m} \sin(\pi n\rho) \sin(\pi m\rho) c_{n, m}(x).
		\]
		Assume that $u\in L^2((-1, 1)^2)$ with periodic boundary data. Hence, we can represent $u$ by its Fourier series $u = \sum_{n, m \in \N} u^c_{n, m} c_{n, m} + u^s_{n, m} s_{n, m}$. This leads to
		\[
		(-\Laplace)^\alpha u(x) = \sum_{n, m \in \N} \lambda_{n, m}^\alpha u^c_{n, m} c_{n, m}  (x)  +  \lambda_{n, m}^\alpha  u^s_{n, m} s_{n, m} (x) .
		\]
		Comparing coefficients with the Fourier series of $g$, we see that $u^s_{n, m} = 0$ and that 
		\[
		u^c_{n, m} = -\frac{4(F_1 + F_2)}{ \pi^{2+2\alpha} (n^2+m^2)^\alpha n m} \sin(\pi n \rho) \sin(\pi m \rho).
		\]
		This proves the lemma.
	\end{proof}
	\begin{Lemma} \label{kor: l infinty bound on solution of fraclp u = g}
		The function $u$ from lemma \ref{lem: solution fraclp u = g} has the following $L^\infty$ bounds depending on $\alpha$,
		\[
		||u||_{\infty} \le C(\alpha) (F_1 + F_2) \rho^{2\alpha},
		\]
		where $C(\alpha)$ is a constant depending badly on $\alpha$. This means in our case, that $C(\alpha) \to \infty$ as $\alpha \to 0$ and $\alpha \to 1$.
	\end{Lemma}
	\begin{proof}
		As $n^2+m^2 \ge 2nm$ implies that $(n^2+m^2)^\alpha \ge (2nm)^\alpha$, we can estimate
		\begin{align*}
		|u(x)| &= \frac{4(F_1+F_2)}{\pi^{2+2\alpha}} \left|\sum_{n, m \in \N} \frac{\sin(\pi n\rho)\sin(\pi m\rho)\cos(\pi x n + \pi y m)}{(n^2+m^2)^\alpha nm}\right| \\
		&\le  \frac{4(F_1+F_2)}{\pi^{2+2\alpha}} \sum_{n, m \in \N} \frac{|\sin(\pi n\rho)\sin(\pi m\rho)\cos(\pi x n + \pi y m)|}{n^{1+\alpha}m^{1+\alpha}}.
		\end{align*}
		We note that $|\cos(\pi xn + \pi xm)| \le 1$ and hence
		\begin{align*}
		|u(x)| &\le  \frac{4(F_1+F_2)}{\pi^{2+2\alpha}} \sum_{n\in\N} \frac{|\sin(\pi n\rho)|}{n^{1+\alpha}} \sum_{m\in\N} \frac{|\sin(\pi m\rho)|}{m^{1+\alpha}}\\
		&\le  \frac{4(F_1+F_2)}{\pi^{2+2\alpha}} \left(\sum_{n\in\N} \frac{|\sin(\pi n\rho)|}{n^{1+\alpha}}\right)^2.
		\end{align*}
		Let us now estimate the sum, using $|\sin(\pi n\rho)| \le \max\{1, \pi n\rho\}$,
		\begin{align*}
		\sum_{n\in\N} \frac{|\sin(\pi \rho n)|}{n^{1+\alpha}} 
		&\le \pi \rho + \int_{1}^{(2\rho)^{-1}} \frac{\pi \rho }{n^{\alpha}} \;\mathrm{d}n + \int_{(2\rho)^{-1}}^\infty \frac{1}{n^{1+\alpha}} \;\mathrm{d}n \\
		&= \pi \rho + \pi \rho (1-\alpha)^{-1} (2\rho)^{\alpha - 1} - \pi  \rho (1-\alpha)^{-1} + \alpha^{-1} (2\rho)^{\alpha} \\
		&\le C(\alpha) \rho^\alpha,
		\end{align*}
		and the result follows.
	\end{proof}
	\begin{Lemma}\label{lem: req props}
		Let $\mu > 0$  and $F_0 = \mu \rho^2 $, then there exists a solution $w_0 : \mathbb{T}^2 \to \R$ to \eqref{eq: abstract simplif} with $A = (-\Delta)^\alpha$ and $N = 1$ and it holds
		\[
			||w_0||_\infty \le C(\alpha) \mu \rho^{2\alpha}.
		\]
		If $\mu < \frac{1}{C(\alpha)} \rho^{1-2\alpha}$ then
		\[
			||w_0||_\infty < \rho.
		\]
		If $\mu < \frac{1}{C(\alpha)} \rho^{1-{2\alpha}} (\rho^{-\beta} - 1)$, then
		\[
			\max_{x\in \mathbb{T}^n}\{ w_0(x)\} - \min_{x\in\mathbb{T}^n}\{ w_0(x)\} < 2\rho^{1-\beta} -2 \rho. 
		\]
	\end{Lemma}
	\begin{proof}
		Let $\overline{g}(x) \coloneqq -\mu \chi_{[-\rho,\rho]^2}(x) + F_0$, then $g$ has zero average over $[-1, 1]^2$ and we can apply Lemma \ref{lem: solution fraclp u = g} with
		\begin{align*}
		F_2 &\coloneqq F_0 = \rho^2 \mu > 0,\\
		F_1 &\coloneqq \mu - F_0 = \mu - \rho^2\mu > 0.
		\end{align*}
		Therefore there exists a solution $w_0$, by lemma \ref{lem: solution fraclp u = g}, that is bounded, see lemma \ref{kor: l infinty bound on solution of fraclp u = g}.
		The two estimates follow by elementary computations, note that we estimate $\max_{x\in \mathbb{T}^n}\{ w_0(x)\} - \min_{x\in\mathbb{T}^n}\{ w_0(x)\}$ by $2||w_0||_\infty$.
	\end{proof}
	
	Now, we have all the information we need to proof theorem \ref{satz: critical force twin}.
	
	\begin{proof}[Proof of theorem \ref{satz: critical force twin}]
		From now on, we will assume that $\alpha = \frac{1}{2}$, however the statements hold also true for all $\alpha \ge \frac{1}{2}$.
		Let us start by deriving the right criteria for the existence of a stationary supersolution. In this case, we have to choose $\rho = \lambda R$. Moreover, we obtain two conditions on $\mu$, namely $\mu \le \varphi_*$ and $\mu <\tfrac{1}{C} \coloneqq \tfrac{1}{C(2^{-1})}$. Hence, we choose $\mu = \min\{ \varphi_*, \frac{1}{2C} \}$ and then obtain that for all $F < F_0 = \min\{ \varphi_*, \frac{1}{2C} \}(\lambda R)^2$, equation \eqref{eq: driving equation twin boundary} has a stationary supersolution by lemma \ref{lem: req props} and proposition \ref{prop: 3.1}. Therefore, we obtain $\overline{F}_{twin} \ge \min\{ \varphi_*, \frac{1}{2C} \}(\lambda R)^2$.
		
		For the upper bound, we can again use lemma \ref{lem: req props} combined with proposition \ref{prop: 3.2}. This time, the statement follows if
		$$
			\varphi^* \le \mu \le \frac{1}{C} (R^{-\beta} - 1).
		$$
		As $\beta \in (0, 1)$ the right hand-side grows to $+\infty$ as $R\to 0$ and hence, for $R$ small enough we can choose $\mu = \varphi^*$. In conclusion, it follows $\underline{F}_{twin} \le \varphi^*R^2$.
	\end{proof}

\section{Scaling Results for $(n+1)$-dimensional QEW-1 equations}\label{sec: qew}
	Let us consider the $(n+1)$-dimensional QEW-1 equations, i.e.,
	\begin{equation}\label{eq: qew}
		\partial_t u(x ,t) = \Delta u(x, t) - \varphi(x, u(x, t)) + F \text{ in } \mathbb{T}^n \times [0, \infty),
	\end{equation}
	where $\Delta u \coloneqq \sum_{i=1}^n \partial_{x_i x_i}u$ is the Laplacian, $\varphi : \mathbb{T}^n \times \R \to \R$ satisfies the $(n+1)$-dimensional analog of assumptions \ref{ass: distribution of precipitates}, \ref{ass: shape of precipitates} and is Lipschitz-continuous. Moreover, $F > 0$ is a positive external driving force.
	
	As this equation satisfies a comparison principle, we can similarly to proposition \ref{prop: ex pinning forces} derive the existence of critical pinning forces $\underline{F}$ and $\overline{F}$ with
	\begin{itemize}
		\item For all $F < \underline{F}$, the interface $\Gamma(t) \coloneqq \{(x, u(x, t)) \;|\; x\in \Rn\}$ gets pinned.
		\item For all $F > \overline{F}$, the interface $\Gamma(t) \coloneqq \{(x, u(x, t)) \;|\; x\in \Rn\}$ propagates to $+\infty$.
	\end{itemize}
	
	\begin{Theorem}
		For $R > 0$ small enough, the critical pinning forces $\underline{F}, \overline{F}$ for equation \eqref{eq: qew} can almost surely be estimated by
		\[
		\underline{F} \le \varphi^* R^n ~~~{ and }~~~
		\overline{F} \ge \min\{ \varphi_*, (\lambda R)^{1-n} \} \cdot (\lambda R)^n.
		\]
	\end{Theorem}

	\begin{proof}
		The proof is very similar to the proofs from the last section, i.e., we want to use proposition \ref{prop: 3.1} and \ref{prop: 3.2}. Similar to lemma \ref{lem: disloc aux}, we define for $\mu > F_0 > 0$ and $\rho > 0$ the function
		\[
			u_0(x) \coloneqq \left\{\begin{array}{ll}
			\frac{\mu-F_0}{2n}(|x|^2-\rho^2) &\text{if $x\in[-\rho, \rho]^n$,} \\
			\frac{F_0}{2n}((1-\rho)^2 - (1-|x|)^2) &\text{elsewhere.}
			\end{array} \right.
		\]
		The function $u_0$ statisfies equation \eqref{eq: abstract simplif} pointwise with $N = 1$ and $A = \Delta$. Moreover, if $F_0 = \mu \rho^n$, then $u_0$ is a continuous weak solution and hence also a viscosity solution.
		
		We now discuss the lower bound. A simple computation shows that $|u_0| < \rho$, if
		\[
			F_0 < \min\{ \mu+ 2n\rho^{-1}, 2n \tfrac{\rho}{(1-\rho)^2} \}.
		\]
		Therefore, $F_0$ exists if $\mu \le \rho^{1-n}$ and we can apply proposition \ref{prop: 3.1} with $\rho = \lambda R $, $\Phi =\varphi$, $\mu = \min\{\varphi_*, (\lambda R)^{1-n} \}$ and $F_0 = \mu \rho^n$.
		
		For the upper bound, we compute that $\max\{ u_0 \} - \min \{ u_0 \} < 2\rho^{1-\beta} - 2\rho$ if
		\[
			F_0 < \frac{4n \rho^{1-\beta}-4n\rho-\mu\rho^2}{1-2\rho}.
		\]
		Again, such an $F_0$ exists if
		\[
			\mu \rho^{n-1} < \frac{4n \rho^{-\beta}-4n-\mu\rho}{1-2\rho}
		\]
		which is true for $\rho$ small enough as the left hand-side remains bounded and the right hand-side goes to $+\infty$. Therefore, we can apply proposition \ref{prop: 3.2} with $\rho = R$, $\Phi =\varphi$, $\mu = \varphi^*$ and $F_0 = \mu \rho^n$ and the result follows.
	\end{proof}

	\section{Conclusion and Future Work}
In this work we have considered dislocations and twin boundaries propagating through a medium containing precipitates.  We have shown that in materials with well-spaced, quasi periodic arrangement of spherical precipitates (or precipitates that are bounded from inside and outside by spheres), the critical pinning force for dislocations scales, i.e., is bounded rigorously from both above and below by terms that scale as the radius of the precipitates while that of twin boundaries scales as the square of the radius of the precipitates.  It follows that dislocations are more likely to get pinned than twin boundaries in crystals with a well spaced, quasi periodic arrangement of spherical precipitates. Future work would entail looking at more general arrangements.  The key technical difficulty in obtaining such a result is constructing non-stationary subsolutions to a random arrangement of precipitates (for instance generated by a Poisson process).

	\subsection*{Acknowledgments}
	LC and KB acknowledge the support of the the Army Research Laboratory under Cooperative Agreement Number W911NF-12-2-0022. The views and conclusions contained in this document are those of the authors and should not be interpreted as representing the official policies, either expressed or implied, of the Army Research Laboratory or the U.S. Government. The U.S. Government is authorized to reproduce and distribute reprints for Government purposes notwithstanding any copyright notation herein. The opportunity to conduct this research has been made possible by the SURF program of the California Institute of Technology. PWD acknowledges partial funding by the German Scholars Organization/Carl-Zeiss-Stiftung in the form of the ``Wissenschaftler-R{\"u}ckkehrprogramm.''
	
	\subsection*{Conflict of Interest}
	The authors declare no conflict of interest.
	
	\bibliographystyle{abbrv}

\end{document}